\documentclass[11pt,letterpaper]{amsart}
\usepackage{amssymb,amsfonts, amsmath, color,graphicx, comment, mathtools, amsthm}
\usepackage{fontawesome}
\usepackage{tikz, tikz-cd, qtree} 

\usepackage{enumerate}

\usepackage[letterpaper,margin=1in]{geometry}

\usepackage{amsthm}
\usepackage[hidelinks]{hyperref}

\newtheorem{main}{Theorem}

\newtheorem{thm}{Theorem}

\newtheorem{lem}[thm]{Lemma}

\newtheorem{thm*}{Theorem}

\theoremstyle{definition}

\newtheorem{fact}[thm]{Fact}

\theoremstyle{remark}

\numberwithin{equation}{section}

\newcommand{\Z}{\mathbb{Z}}
\newcommand{\R}{\mathbb{R}}

\newcommand{\C}{\mathbb{C}}
\newcommand{\N}{\mathbb{N}}

\newcommand{\cU}{\mathcal{U}}

\renewcommand{\phi}{\varphi}

\newcommand{\norm}[1]{\left\lVert #1 \right\rVert}

\DeclareMathOperator{\Haar}{Haar}

\begin{document}

\title{The unitary group of a $\mathrm{II}_1$ factor is SOT-contractible}
\author{David Jekel}

\address{\parbox{\linewidth}{Department of
		Mathematical Sciences,
		University of Copenhagen \\
		Universitetsparken 5, 2100 Copenhagen \O, Denmark}}
\email{daj@math.ku.dk}
\urladdr{http://davidjekel.com}

\begin{abstract}
We show that the unitary group of any SOT-separable $\mathrm{II}_1$ factor $M$, with the strong operator topology, is contractible.  Combined with several old results, this implies that the same is true for any SOT-separable von Neumann algebra with no type $\mathrm{I}_n$ direct summands ($n < \infty$).  The proof for the $\mathrm{II}_1$-factor case uses regularization via free convolution and Popa's theorem on the existence of approximately free Haar unitaries in $\mathrm{II}_1$ factors.
\end{abstract}

\maketitle

\section{Introduction}

It is natural and old question to determine the homotopy type of the unitary group of a von Neumann algebra.  Dixmier and Douady showed that for a separable infinite-dimensional Hilbert space $H$, the unitary group of $B(H)$ is contractible with respect to the strong operator topology (SOT) \cite{DixDou1963}.  Two years later Kuiper showed it was contractible even with respect to the operator norm topology \cite{Kuiper1965}.  The same is true for any properly infinite von Neumann algebra thanks to \cite{Breuer1967,BrWi1976}.  The hardest case turns out to be that of $\mathrm{II}_1$ factors, for which the unitary group with respect to operator norm has nontrivial $\pi_1$, and hence is not contractible \cite{ASS1971}.

However, when we use the strong operator topology (SOT), Popa and Takesaki \cite{PoTa1993} used Tomita--Takesaki theory to show that the unitary group is contractible for any $\mathrm{II}_1$ factor that is \emph{McDuff}, or stable under tensorization by the hyperfinite $\mathrm{II}_1$ factor $\mathcal{R}$.  Ozawa \cite{Ozawa2024} recently showed that the group of approximately inner automorphisms is also contractible in this setting.  This paper will show SOT-contractibility of the unitary group for all SOT-separable $\mathrm{II}_1$ factors.

\begin{main} \label{thm: contractible}
For any SOT-separable $\mathrm{II}_1$ factor $M$, the unitary group $U(M)$ is SOT-contractible.
\end{main}

The proof has two main stages:  We first deform the unitaries to have spectral distribution close to the Haar measure on $S^1$, and then we deform them toward the identity by functional calculus.  For the first step, the idea is that if we multiply $u$ by a freely independent Haar unitary $v$ (which of course is SOT-homotopic to $1$), then $uv$ has the Haar distribution.  Using Popa's theorem on approximately free unitaries \cite{Popa1995free}, we can find a unitary $v$ freely independent from $M$ in the ultrapower $M^{\cU}$.  We then lift $v$ (and the path from $v$ to $1$) to unitaries in $M$ that provide better and better approximate freeness with each unitary $u_j$ in an SOT-dense sequence in $U(M)$.

Theorem \ref{thm: contractible} in turn allows us to determine the homotopy type with respect to SOT for the unitary group of any SOT-separable von Neumann algebra, via von Neumann's direct integral decomposition \cite{vonNeumann1949reduction}.  Indeed,
\[
M \cong M_0 \oplus \bigoplus_{i \in I} M_i,
\]
where $M_0$ has diffuse center and $M_i$ is a factor for $i \in I$.  It is easy to see that $U(M_0)$ is SOT-contractible.  Indeed, writing $Z(M_0) \cong L^\infty[0,1]$, consider the central projection $1_{[0,t]}$ in $L^\infty[0,1]$ which depends SOT-continuously on $t$, and define the homotopy
\[
h(t,u) = 1_{[0,t]} + 1_{(t,1]} u, \qquad t \in [0,1], u \in U(M_0).
\]
Moreover, by the previously mentioned results for properly infinite von Neumann algebras, the unitary groups of the type $\mathrm{I}_\infty$, $\mathrm{II}_\infty$ and $\mathrm{III}$ summands are contractible, and by Theorem \ref{thm: contractible}, the same holds for the type $\mathrm{II}_1$ summands. By taking the direct sum of all these homotopies, and the trivial homotopy on the type $\mathrm{I}_n$ summands, we obtain a homotopy from $U(M)$ to the unitary group of the type $\mathrm{I}_n$ ($n < \infty$) summands.  Thus, we have the following.

\begin{main}
For any SOT-separable von Neumann algebra $M$, the unitary group is SOT-homotopic to the unitary group of its atomic, type $\mathrm{I}_n$ ($n < \infty$) part.  In particular, the unitary group is SOT-contractible if and only if $M$ has no type $\mathrm{I}_n$ ($n < \infty$) direct summand.
\end{main}

\subsection*{Acknowledgements}

I thank Narutaka Ozawa for his lecture at the conference $\mathrm{C}^*$-algebras at the Mathematisches Forschungsinstitut Oberwolfach August 4-8, 2025, which explained the problem of contractibility of the unitary group while giving history and intuition.  I also thank Hannes Thiel for advertising the problem on his website.  I thank Sorin Popa, Hannes Thiel, Ben Hayes, and Mehdi Moradi for comments on the manuscript.  Funding for this work was provided by the Horizon Europe Marie Sk{\l}odowska Curie Action, FREEINFOGEOM, grant 101209517.

\section{Preliminaries}

We assume familiarity with tracial von Neumann algebras and $\mathrm{II}_1$ factors (see e.g.\ \cite{TakesakiI,JonesSunder1997}).  In particular, recall that on a tracial von Neumann algebra, the SOT is metrized by the $2$-norm $\norm{x}_2 = \tau(x^*x)^{1/2}$.  We will also use the ultrapower $M^{\cU}$ for a free ultrafilter $\cU$ on $\N$ (for background, see e.g.\ \cite[Appendix E]{BrownOzawa2008}).

We first recall some terminology for measures on the circle and the Wasserstein distance.  Let $S^1 = \{z \in \C: |z| = 1\}$.  If $u$ is a unitary in $M$, then by the spectral theorem, there is a unique probability measure $\mu_u \in \mathcal{P}(S^1)$ such that
\[
\tau(f(u)) = \int_{S^1} f(z)\,d\mu_u(z) \text{ for } f \in C(S^1),
\]
and by the Stone-{-Weierstrass theorem, $\mu_u$ is uniquely characterized by
\[
\tau(u^k) = \int_{S^1} z^k \,d\mu_u(z) \text{ for } k \in \Z.
\]
We call $\mu_u$ the spectral measure or distribution of $u$ (with respect to the trace).  We denote by $\mu_{\Haar}$ the Haar measure, or uniform probability measure, on $S^1$.

The space of probability measures $\mathcal{P}(S^1)$ will be equipped with the weak-$*$ topology.  We also recall the definition of the $L^2$ Wasserstein distance in the free setting as in \cite{BV2001}, but for the unitary case:
\[
d_W(\mu,\nu) = \inf \{\norm{u - v}_{L^2(M,\tau)}: \mu_u = \mu, \mu_v = \nu: (M,\tau) \text{ separable tracial von Neumann algebra} \}.
\]
Although the class of separable tracial von Neumann algebras is not a set, the above infimum can equivalently be expressed as one over a set (see e.g.\ \cite[Lemma 2.20]{GJNS2021}).  In fact, since a single unitary generates a commutative algebra, (the unitary analog of) \cite[Theorem 1.5]{BV2001} shows that the free Wasserstein distance for this case can always be witnessed by commutative algebras $(M,\tau)$ and so agrees with the classical Wasserstein distance.

\begin{fact}[See {e.g.~\cite[Theorem 6.9]{Villani2008}}] \label{fact: topologies on measures}
The weak-$*$ topology on $\mathcal{P}(S^1)$ agrees with the $d_W$-topology.
\end{fact}

\begin{fact} \label{fact: ultraproduct measure convergence}
Let $u_n$ be a sequence of unitaries in $M$, and let $u = [u_n]_{n \in \N}$ be the corresponding element in $M^{\cU}$.  Then $\mu_u = \lim_{n \to \cU} \mu_{u_n}$.
\end{fact}

\begin{proof}
By construction of the multiplication, adjoint, and trace on the ultraproduct, for $k \in \N$,
\[
\lim_{n \to \cU} \int_{S^1} z^k \,d\mu_{u_n}(z) = \lim_{n \to \cU} \tau_M(u_n^k) = \tau_{M^{\cU}}(u) = \int_{S^1} z^k\,d\mu_u(z).
\]
Since polynomials span a dense subset of $C(S^1)$, we have $\mu_{u_n} \to \mu$ weak-$*$.
\end{proof}

Next, we recall the notion of \emph{free independence} (see e.g.\ \cite[Definition 2.5.1]{VDN1992}).  Let $(M,\tau)$ be a tracial von Neumann algebra.  Then von Neumann subalgebras $(M_i)_{i \in I}$ are said to be freely independent if for every $k \in \N$, for any $i_1 \neq i_2 \neq \dots \neq i_k$, for any $x_1 \in M_{i_1}$, \dots, $x_k \in M_{i_k}$, we have
\[
\tau[(x_1 - \tau(x_1)) \dots (x_k - \tau(x_k))] = 0.
\]
With free independence, the value of $\tau$ on any product of elements from the $M_i$'s is uniquely determined by the $\tau|_{M_i}$'s \cite[Proposition 2.5.5(1)]{VDN1992}.  Moreover, the free product construction allows realization of freely independent copies of any given tracial von Neumann algebras \cite[Proposition 1.5.5]{VDN1992}.  Two elements or tuples from $M$ are said to be freely independent if the von Neumann subalgebras that they generate are freely independent.

If $u$ and $v$ are freely independent unitaries, then the measure $\mu_{uv}$ is uniquely determined by $\mu_u$ and $\mu_v$ in light of \cite[Proposition 2.5.5(1)]{VDN1992}.  The measure $\mu_{uv}$ is called the \emph{free convolution} of $\mu_u$ and $\mu_v$ and denoted $\mu_u \boxtimes \mu_v$.  Note that for any two measures $\mu$ and $\nu$ in $\mathcal{P}(S^1)$, there exist some freely independent unitaries with spectral distributions $\mu$ and $\nu$, and hence $\mu \boxtimes \nu$ is well-defined.

\begin{fact} \label{fact: free convolution Haar}
If either $\mu$ or $\nu$ is $\mu_{\Haar}$, then $\mu \boxtimes \nu = \mu_{\Haar}$.
\end{fact}

\begin{proof}
Let $u$ and $v$ be freely independent with $\mu_u = \mu_{\Haar}$.  For $z \in S^1$, note that $zu$ is also freely independent of $v$ and has $\mu_{zu} = \mu_{\Haar}$.  Therefore, by free independence, $(zu)v$ and $uv$ have the same distribution, so $\mu_{uv}$ is rotation-invariant and thus equals $\mu_{\Haar}$.  The case when $\mu_v = \mu_{\Haar}$ is symmetrical.
\end{proof}

\begin{fact} \label{fact: free convolution Wasserstein}
Let $\mu_1$, $\mu_2$, $\nu \in \mathcal{P}(S^1)$.  Then $d_W(\mu_1 \boxtimes \nu, \mu_2 \boxtimes \nu) \leq d_W(\mu_1,\mu_2)$.
\end{fact}

\begin{proof}
Let $u_1$ and $u_2$ be unitaries in some tracial von Neumann algebra $(M,\tau_M)$ with the spectral measures $\mu_1$ and $\mu_2$ respectively.  Let $(N,\tau_N)$ be a larger tracial von Neumann algebra containing a unitary $v$ freely independent from $M$ with $\mu_v = \nu$.  Then $\mu_{u_1v} = \mu_1 \boxtimes \nu$ and $\mu_{u_2v} = \mu_1 \boxtimes \nu$ and $\norm{u_1v - u_2v}_2 = \norm{u_1 - u_2}_2$.  Since the initial choice of $M$ and $u_1$ and $u_2$ with spectral measures $\mu_1$ and $\mu_2$ was arbitrary, the claim follows.
\end{proof}

Finally, we recall Popa's remarkable theorem on approximately free elements in $\mathrm{II}_1$ factors.  This is the most significant ingredient that makes our proof of Theorem \ref{thm: contractible} possible.

\begin{thm}[Popa {\cite{Popa1995free}}] \label{thm: Popa freeness}
Let $M$ be a separable $\mathrm{II}_1$ factor and $\cU$ a free ultrafilter on $\N$.  View $M \subseteq M^{\cU}$ via the diagonal embedding.  Then there exists a unitary $v \in M^{\cU}$ freely independent of $M$, with $\mu_v = \mu_{\Haar}$.
\end{thm}

\section{Proof of the main theorem}

The first step of the proof is to homotope the unitaries so that their spectral measures converge to the Haar measure $\mu_{\Haar}$ as $t \to \infty$.

\begin{lem} \label{lem: Haar homotopy}
Let $M$ be a separable $\mathrm{II}_1$ factor.  Then there exists a map $h: [0,\infty) \times U(M) \to U(M)$ with the following properties.
\begin{enumerate}[(1)]
	\item $\norm{h(t,u) - h(s,v)}_2 \leq \pi|t-s| + \norm{u-v}_2$.
	\item For each $u \in U(M)$, we have $\lim_{t \to \infty} d_W(\mu_{h(t,u)},\mu_{\Haar}) = 0$.
\end{enumerate}
\end{lem}

\begin{proof}
Consider $M \subseteq M^{\cU}$ and let $v$ be a freely independent Haar unitary given by Theorem \ref{thm: Popa freeness}.  Write $v = e^{\pi i x}$ where $x \in M^{\cU}$ is self-adjoint with $\norm{x} \leq 1$.  Let $x = [x_n]_{n \in \N}$ with $x_n$ self-adjoint and $\norm{x_n} \leq 1$.

Let $(u_m)_{m \in \N}$ enumerate a dense subset of $U(M)$.  We will define $h(t,u)$ for $t \in [m,m+1]$ by induction on $m$.  In fact to define it on $[m,m+1]$, we only need to know the value at $t = m$, and so the base case where $m = 0$ proceeds in the same way as the induction step.  So let $m \geq 0$ and assume $h$ is defined on $[0,m] \times U(M)$ with $h(0,u) = u$.

For $j = 1$, \dots, $m$, and $s \in [0,1]$, we have that $h(m,u_j)$ is free from $e^{\pi i s x}$ in $M^{\cU}$.  By Facts \ref{fact: topologies on measures} and \ref{fact: ultraproduct measure convergence}, we have
\[
\forall s \in [0,1], \quad \lim_{n \to \cU} d_W(\mu_{h(m,u_j) e^{\pi i sx_n}},\mu_{h(m,u_j) e^{\pi i sx}}) = 0.
\]
Furthermore,
\[
s \mapsto d_W(\mu_{h(m,u_j) e^{\pi i sx}},\mu_{h(m,u_j) e^{\pi i sx}}) \text{ is $2\pi$-Lipschitz}.
\]
Therefore, since equicontinuity and pointwise convergence imply uniform convergence,
\[
\lim_{n \to \cU} d_W(\mu_{h(m,u_j) e^{\pi i sx_n}},\mu_{h(m,u_j) e^{\pi i sx}}) = 0 \text{ uniformly for } s \in [0,1].
\]
Hence, choose $n(m)$ such that
\[
\max_{j \leq m} \sup_{s \in [0,1]} d_W(\mu_{h(m,u_j) e^{\pi i sx_{n(m)}}},\mu_{h(m,u_j) e^{\pi i sx}}) \leq \frac{1}{2^m}.
\]
We then define
\[
h(m+s,u) = h(m,u) e^{\pi i sx_{n(m)}}\text{ for } s \in [0,1].
\]
This completes the inductive definition.

To check claim (1), note that by construction $s \mapsto h(m+s,u)$ is $\pi$-Lipschitz for $s \in [0,1]$ and for $m \in \N$, which implies it is Lipschitz in the time variable on all of $[0,\infty)$.  The Lipschitzness in $u$ follows because by construction $h(t,u) = uv_t$ for some unitary $v_t$ independent of $u$.

Now we prove (2).  By Facts \ref{fact: free convolution Haar} and \ref{fact: free convolution Wasserstein}, we have
\[
\mu_{h(m,u_j) e^{\pi i sx}} = \mu_{h(m,u_j)} \boxtimes \mu_{\Haar} = \mu_{\Haar},
\]
and hence by our choice of $x_{n(m)}$, we have
\[
\forall j \leq m, \quad d_W(\mu_{h(m+1,u_j)}, \mu_{\Haar}) = d_W(\mu_{h(m,u_j)e^{\pi i x_{n(m)}}}, \mu_{h(m,u_j)e^{\pi i x}}) \leq \frac{1}{2^m}.
\]
Furthermore, for $s \in [0,1]$, 
\[
d_W(\mu_{h(m,u_j) e^{\pi i sx}}, \mu_{\Haar}) \leq d_W(\mu_{h(m,u_j)},\mu_{\Haar}),
\]
and hence
\[
d_W(\mu_{h(m+s,u_j)},\mu_{\Haar}) = d_W(\mu_{h(m,u_j) e^{\pi i sx_{n(m)}}}, \mu_{\Haar}) \leq \frac{1}{2^m} + d_W(\mu_{h(m,u_j) e^{\pi i sx}}, \mu_{\Haar}).
\]
Putting these estimates together,
\[
\forall m \geq j, \sup_{s \in [0,1]} d_W(\mu_{h(m+1+s,u_j)},\mu_{\Haar}) \leq \frac{1}{2^m} + \frac{1}{2^{m+1}}.
\]
In particular,
\[
\forall m \geq j+1, \quad \sup_{t \geq m} d_W(\mu_{h(t,u_j)},\mu_{\Haar}) \leq \frac{3}{2^m},
\]
which proves claim (2) in the case when $u$ is one of the elements $u_j$.  Because $(u_j)_{j \in \N}$ is dense in $U(M)$ in the $2$-norm, and because $h(t,u)$ is $1$-Lipschitz in $u$, the claim extends to all of $U(M)$.
\end{proof}

For the second step of the proof, after the unitaries are homotoped to have spectral measures sufficiently close to $\mu_{\Haar}$, we apply a deformation by functional calculus to push the mass of the measure towards the point $1$ on the circle.

\begin{lem} \label{lem: functional deformation}
Let $f: [0,\infty) \times \R/\Z \to \R/\Z$ be given by $f(t,x) = f_t(x) = x|2x|^t$ for $x \in (-1/2,1/2]$.  Since the values agree at $\pm 1/2$, $f_t$ is well-defined and continuous on $\R/\Z$.  Hence, let $g_t(z) = g(t,z)$ be the function $[0,\infty) \times S^1 \to S^1$ given by $g_t(e^{2\pi i x}) = e^{2\pi i f_t(x)}$.  The for any unitary $u$ in a tracial von Neumann algebra,
\[
\norm{g_t(u) - 1}_2 \leq (2t+1)d_W(\mu_u, \mu_{\Haar}) + \sqrt{\frac{\pi}{2(t+1)}}.
\]
\end{lem}

\begin{proof}
Remark that
\begin{multline*}
\int_{S^1} |g_t(z) - 1|^2\,d\mu_{\Haar} \leq 2\pi \int_{-1/2}^{1/2} |f_t(x)|^2\,dx = 2^{2t+1} \pi \int_{-1/2}^{1/2} |x|^{2t + 1} \,dx \\ = 2^{2t+1}\pi \cdot \frac{1}{2t+2} 2(1/2)^{2t+2} = \frac{\pi}{2(t+1)}.
\end{multline*}
Hence,
\[
d_W((g_t)_* \mu_{\Haar}, \delta_1) \leq \sqrt{\frac{\pi}{2(t+1)}}.
\]
Meanwhile, note that $|f_t'| \leq 2t+1$-Lipschitz on $[-1/2,1/2]$, and hence also the derivative of $g_t$ along the circle is bounded by $2t+1$, and it is $2t+1$-Lipschitz.  Since the Wasserstein distance can be witnessed in classical probability spaces, it follows that
\[
d_W((g_t)_* \mu_u, (g_t)_* \mu_{\Haar}) \leq (2t+1) d_W(\mu_u, \mu_{\Haar}).
\]
Hence, by the triangle inequality,
\[
d_W((g_t)_* \mu_u, \delta_1) \leq (2t+1)d_W(\mu_u, \mu_{\Haar}) + \sqrt{\frac{\pi}{2(t+1)}}.
\]
Note that the $L^2$-distance of a unitary $v$ from $1$ is equal to the Wasserstein distance of $\mu_v$ from $\delta_1$.
\end{proof}

\begin{proof}[Proof of Theorem \ref{thm: contractible}]
Let $h(t,u)$ be as in Lemma \ref{lem: Haar homotopy}.  Then $g_t$ be as in Lemma \ref{lem: functional deformation}.  Let
\[
s(t,u) = \min(d_W(\mu_u,\mu_{\Haar})^{-1/2},t),
\]
which is jointly continuous in $(t,u)$.  For $t \in [0,\infty)$ and $u \in U(M)$, let
\[
\tilde{h}(t,u) = g_{s(t,h(t,u)))}(h(t,u)),
\]
which is jointly continuous in $(t,u)$ by construction.  By Lemma \ref{lem: functional deformation}, we have
\begin{align*}
\norm{\tilde{h}(t,u) - 1}_2 &\leq (2s(t,u) + 1) d_W(\mu_{h(t,u)},\mu_{\Haar}) + \sqrt{\frac{\pi}{2s(t,u) + 1}} \\
&\leq (2d_W(\mu_{h(t,u)},\mu_{\Haar})^{-1/2} + 1) d_W(\mu_{h(t,u)},\mu_{\Haar}) \\
&\quad + \max\left( \sqrt{\frac{\pi}{2(d_W(\mu_{h(t,u)},\mu_{\Haar})^{-1/2} + 1)}}, \sqrt{\frac{\pi}{2(t+1)}} \right) \\
&\leq C_1 d_W(\mu_{h(t,u)},\mu_{\Haar})^{1/4} + C_2t^{-1/2}
\end{align*}
for universal constants $C_1$ and $C_2$ (note that the Wasserstein distance on $S^1$ is always bounded by the diameter $2$).  This shows that $\tilde{h}(t,u)$ is continuous and satisfies
\[
\lim_{t \to \infty} \norm{\tilde{h}(t,u) - 1}_2 = 0.
\]
Moreover, because function $d_W(\mu_{h(t,u)},\mu_{\Haar})$, used in the upper bound, is $1$-Lipschitz in $u$, we see that
\begin{multline*}
\forall u \in U(M), \forall \epsilon > 0, \exists \delta > 0, \exists T > 0 \text{ such that} \\ \forall t, \forall v \in U(M), \norm{u - v}_2 < \delta \text{ and } t \geq T \implies \norm{\tilde{h}(t,u) - 1}_2 < \epsilon.
\end{multline*}
In other words, if we set $\tilde{h}(\infty,u) = 1$, then $\tilde{h}$ is continuous on $[0,\infty] \times U(M)$.  This produces the desired homotopy from the identity on $U(M)$ to the constant function $1$.
\end{proof}

\subsection*{Declarations}

The author states that there is no conflict of interest.  The manuscript has no associated data.

\bibliographystyle{plain}
\bibliography{contractible-unitary-references}

\end{document}